\documentclass[11pt]{amsart}
\usepackage[utf8]{inputenc}
\usepackage{amsmath,amsfonts,amssymb}

\usepackage{enumerate}
\usepackage{mathrsfs}
\usepackage{graphicx}

\usepackage{listings}
\usepackage{caption} 
\usepackage{hyperref}

\usepackage[ruled,vlined,noresetcount]{algorithm2e}
\usepackage{algorithmicx}
\usepackage{pgf,tikz,tikz-3dplot}

\usepackage[margin=1in]{geometry}
\usetikzlibrary{arrows.meta}
\usetikzlibrary{graphs, graphs.standard}
\tikzgraphsset{declare={polygon_n}{[clique]\foreach\x in\tikzgraphV{\x/}}}

\usepackage{subcaption}
\usepackage{float}

\usetikzlibrary{positioning}
\tikzset{sgplattice/.style={inner sep=1pt,norm/.style={red!50!blue},char/.style={blue!50!black},
  lin/.style={black!50}},cnj/.style={black!50,yshift=-2.5pt,left=-1pt of #1,scale=0.5,fill=white}}

\usepackage{float} 

  \newtheorem{theorem}{Theorem}[section]
  \newtheorem{corollary}[theorem]{Corollary} 
  \newtheorem{lemma}[theorem]{Lemma}
  \newtheorem{proposition}[theorem]{Proposition}
  
  \newtheorem{example}[theorem]{Example}
  \newtheorem{remark}[theorem]{Remark}
  
  \numberwithin{equation}{section}

\def\lcm{\mathrm{lcm}}
\def\rad{\mathrm{rad}}

\newcommand{\gen}[1]{\ensuremath{\langle #1\rangle}}

\title{On the factorization graph  of  finite groups}

\dedicatory{This paper is dedicated to Prof. Dr. Bernhard Amberg}

\author{Ismael Gutierrez*}
\address{* Corresponding Author, Department of Mathematics and Statistics, Universidad del Norte, Km 5 via a Puerto Colombia, Barranquilla - Colombia}
\email{isgutier@uninorte.edu.co}

\author{Yesneri Zuleta}
\address{Department of Mathematics and Statistics, Universidad del Norte, Km 5 via a Puerto Colombia, Barranquilla - Colombia}
\email{zuletay@uninorte.edu.co}
\thanks{ }

\author{E. J. Garc\'ia-Claro}
 \address{Departamento de Matemáticas, Universidad Autónoma Metropolitana, Unidad Iztapalapa, Código postal 09340, Ciudad de México - México }
  \email{eliasjaviergarcia@gmail.com  }

\thanks{ }

\subjclass[2010]{Primary 05C25; Secondary 20D30, 05C40}

\date{ }

\keywords{Factorization graph, cyclic and dicyclic Groups, generalized Quaternion group}

\begin{document}

\begin{abstract}
Given a finite group $G$, the factorization graph $\Gamma(G)$ has as vertices the nontrivial proper subgroups of $G$, with two distinct vertices $H$ and $K$ adjacent if and only if $HK = G$. 
In this paper, we study structural and combinatorial properties of $\Gamma(G)$ for three families of groups: cyclic groups of square-free order, generalized quaternion groups, and dicyclic 
groups. For cyclic groups $C_{p_1\cdots p_r}$, we describe the global structure of $\Gamma(G)$ and determine its order, degree sequence, center, girth, and diameter. We also characterize completely when $\Gamma(G)$ is claw-free, show that it is non-planar for $r \geq 4$, and prove that after removing all pendant vertices, the resulting induced subgraph is 2-connected; we further compute its exact circumference for small $r$, showing it is Hamiltonian precisely when $r=4, 5$. For generalized quaternion groups, we classify all isolated vertices and prove that the reduced factorization graph has diameter two and a unique central vertex. Moreover, we show that it is planar, has chromatic number~$3$, is Eulerian, and non-Hamiltonian, and that it contains maximal 
induced stars and triangular book subgraphs. Finally, we analyze the dicyclic case and show that the reduced factorization graph is a complete bipartite graph, from which all main invariants 
follow directly. An explicit SageMath algorithm for constructing $\Gamma(G)$ is included.
\end{abstract}


\maketitle


\section{Introduction}

Subgroup-based graphs have become an effective combinatorial framework for encoding algebraic properties of groups into discrete structures. In this approach, subgroups of a given group are represented as vertices, while adjacency relations reflect algebraic interactions such as inclusion, generation, or factorization. Over the last decade, a variety of graph constructions of this type have been introduced and studied, revealing strong connections between group-theoretic structure and classical graph invariants.

Among the most prominent examples are the \emph{co-maximal subgroup graph}, introduced by Akbari, Miraftab, and Nikandish \cite{comaximal2, comaximal3, comaximal1}, the \emph{factorization graph} studied by Farrokhi and Azimi \cite{Farrokhi}, and the more recent \emph{subgroup product graph} \cite{Hayder, Das}. Although these graphs share a common guiding principle—encoding subgroup factorizations via adjacency—they differ substantially in their vertex sets and in the algebraic restrictions imposed on adjacency. As a consequence, they are not equivalent objects, and results obtained in one setting do not automatically extend to the others.

In this paper, we focus on the following construction. Let $G$ be a finite group. The \emph{factorization graph} of $G$, denoted by $\Gamma(G)$, is the simple undirected graph whose vertex set consists of all nontrivial proper subgroups of $G$, where two distinct vertices $H$ and $K$ are adjacent if and only if $HK = G$. This definition places no restrictions on the Frattini subgroup and differs from other notions of factorization or subgroup product graphs appearing in the literature. The graph $\Gamma(G)$ defined above is the unique object studied throughout this article.

Our goal is to describe the global structure of $\Gamma(G)$ for several natural families of finite groups, with particular emphasis on graph-theoretic properties such as connectivity, degree distribution, planarity, Hamiltonicity, and the structure of the center and of maximal induced subgraphs. While partial results on subgroup-based graphs for cyclic, dihedral, and quaternion groups are available, a systematic investigation of the factorization graph $\Gamma(G)$—in the precise sense adopted here—has not been carried out for these families.

The first part of the paper is devoted to cyclic groups of square-free order. In this setting, we show that the factorization graph exhibits a remarkably rigid structure: it contains a unique complete subgraph determined by the prime divisors of the group order, together with pendant vertices giving rise to sunlet and corona-type subgraphs. As a consequence, we derive explicit formulas for the degree sequence, determine when the graph is planar or non-regular, describe its center, and prove the existence of large pancyclic induced subgraphs.

In the second part, we study the factorization graph of generalized quaternion groups. After removing isolated vertices, the resulting reduced graph displays strong structural features: it has diameter 2 and a unique central vertex, contains maximal induced stars and triangular book subgraphs of explicitly determined size, and is planar, Eulerian, $3$-colorable, and non-Hamiltonian.

Finally, we briefly analyze the case of dicyclic groups, showing that their factorization graphs reduce to complete bipartite graphs and that their main graph invariants follow directly from this description.

The remainder of the paper is organized as follows. Section 2 fixes notation and recalls the basic graph-theoretic and group-theoretic concepts needed in the sequel. In Section 3, we establish general adjacency criteria and collect preliminary observations on factorization graphs. Sections 4--6 contain the main results for cyclic groups of square-free order, generalized quaternion groups, and dicyclic groups, respectively. An explicit algorithmic construction of $\Gamma(G)$ is included in the appendix.

\section{Preliminaries}

In this section, we fix notation and recall standard terminology from graph theory and finite group theory that will be used throughout the paper. Every graph is assumed to be \emph{finite}, \emph{simple}, and \emph{undirected}.  A graph is denoted by $\Gamma=(V, E)$, where $V$ is the vertex set and $E$ the edge set.

The \emph{order} and \emph{size} of $\Gamma$ are $|V|$ and $|E|$, respectively. Two vertices $x,y\in V$ are said to be \emph{adjacent}, written $x\sim y$, when $\{x,y\}\in E$. 
A graph is called complete if every pair of distinct vertices is adjacent; the complete graph on $n$ vertices is denoted by $K_n$.

The \emph{degree} of a vertex $x$ is $\deg(x)=|\{y\in V\mid y\sim x\}|$. We write $\Delta(\Gamma) = \max_{x\in V}\deg(x)$ and $\delta(\Gamma)=\min_{x\in V}\deg(x)$ for the maximum and minimum degree, respectively.

A \emph{walk} of length $k$ is a sequence of vertices $x_0 x_1\dots x_k$ such that $x_{i-1}\sim x_{i}$ for every $1\leq i\leq k$. The walk is a \emph{path} when all vertices are distinct. If, in addition, $x_{0}=x_k$ we obtain a \emph{cycle}, denoted by $C_k$. A cycle of length three is a \emph{triangle}. 

The circumference of a graph, denoted by $C(\Gamma)$, is the length of the longest cycle in the graph. If the graph contains no cycle, the circumference is undefined (or set to zero by convention). The girth of an undirected graph $\Gamma$, denoted by $g(\Gamma)$, is the shortest cycle length in the graph. If $\Gamma$ does not contain any cycles ($\Gamma$ is a forest), its girth is defined to be infinite. A graph with a girth of four or more is called triangle-free. 

A chord of a path (cycle) $(v_1,\ldots, v_n)$ is an edge between two vertices of the path (cycle) that is not an edge of the path (cycle). A path (cycle) is chordless if it contains no chords. A graph $\Gamma$ is chordal if each cycle in $\Gamma$ of length at least 4 has at least one chord. 

The \emph{distance} between two vertices $x,y$ is the length $d(x,y)$ of a shortest $x$-$y$ path. If there is no path connecting the two vertices, i.e., if they belong to different connected components, then conventionally, the distance is defined as infinite. The \emph{eccentricity} of a vertex $x$ is 
\[\operatorname{ecc}_{\Gamma}(x) = \max_{y\in V} d(x,y);\]
the \emph{radius} and \emph{diameter} of $\Gamma$ are $\operatorname{rad}(\Gamma) = \min_{x\in V} \operatorname{ecc}{\Gamma}(x)$ and $\operatorname{diam}(\Gamma) = \max_{x\in V} \operatorname{ecc}_{\Gamma}(x)$.
The \emph{center} of $\Gamma$ is the set of vertices with minimum eccentricity.

We say that $\Gamma$ is connected if a path can join any two vertices; otherwise, $\Gamma$ is disconnected. A connected graph without cycles is a tree. A spanning tree of $\Gamma$ is a subgraph that includes all the vertices of $\Gamma$ and is also a tree. 

$\Gamma$ is \emph{$k$‑connected} if $|V|>k$ and removal of any $k-1$ vertices leaves it connected.  The vertex‑connectivity is $\kappa(\Gamma)$. In particular, $\Gamma$ is \emph{2-connected} (or \emph{biconnected}) when $\kappa(\Gamma)\geq 2$. A well-known characterization states that a graph is 2-connected if and only if there exists, for any two vertices, a cycle containing these two vertices.

A \emph{clique} is a complete subgraph; its maximum order is $\omega(\Gamma)$.  An \emph{independent set} is a clique in the complement $\overline{\Gamma}$; its maximum size is denoted by $\alpha(\Gamma)$.

A \emph{total dominating set} of $\Gamma$ is a set $D\subseteq V(\Gamma)$ such that every vertex of
$\Gamma$ is joined by an edge with some vertex of $D$. When used, the total domination number is denoted by
$\gamma_t(\Gamma)$.

\smallskip

\paragraph{\emph{Special graphs.}} We recall some graph families that occur frequently in what follows. 
\begin{enumerate}[i)]
\item The star graph of order $n$ is a tree on $n$ nodes with one node having vertex degree $n-1$ and the other $n-1$ nodes having vertex degree 1. Therefore, the star graph of order $n$ is isomorphic to the complete bipartite graph $K_{1,n-1}$. An \emph{induced star} in a graph $\Gamma$ is an induced subgraph of $\Gamma$ that is isomorphic to $K_{1,k}$. A \emph{maximum induced star} in $\Gamma$ is an induced star with the maximum number of leaves among all such induced stars in $\Gamma$.

\item Let $\Gamma_1$ and $\Gamma_2$ be two graphs on disjoint sets of $n_1$ and $n_2$ vertices, respectively. The corona product of $\Gamma_1$ and $\Gamma_2$, denoted by $\Gamma_1 \odot \Gamma_2$, is defined as the graph obtained by taking one copy of $\Gamma_1$ and $n_1$ copies of $\Gamma_2$, and then adding an edge between the $i$th vertex of $\Gamma_1$ to every vertex in the $i$th copy of $\Gamma_2$. Note that if $\Gamma_j$ has $q_j$ edges, then the corona $\Gamma_1 \odot \Gamma_2$ has $n_1(n_2+1)$ vertices and $q_1+ n_1(q_2+n_2)$ edges. For $r\geq 3$, the corona product $C_r\odot K_1$, denote by $L_{2r}$ is called a Sunlet graph on $2r$ vertices.

\item A sun (or trampoline) is a chordal graph $\Gamma$ on $2n$ vertices for some $n\geq 3$ whose vertex set can be partitioned into two sets, $W =\{w_1,\ldots, w_n\}$, $U = \{u_1,\ldots, u_n\}$, such that $W$ is independent, and for each $i$ and $j$, $w_j$ is adjacent to $u_i$ if and only if $i=j$ or $i \equiv j + 1 \bmod n$. A complete sun graph is a sun $\Gamma$ in which $\Gamma[U]$ is a complete graph. A sun is called odd if its partition sets have odd cardinality.

\item A graph is called claw-free if it contains no induced subgraph isomorphic to the complete bipartite graph $K_{1,3}$.
    
\item A graph of order $n$ is called pancyclic if it contains cycles of all possible lengths, from three to $n$. 
    
\item The \emph{triangular book graph} $B_n$ is a graph with set of vertices $V(B_n) = \{u, v, w_1,\dots, w_n\}$ and edges $E(B_n) = \{(u,v)\} \cup \bigcup_{i=1}^{n} \{(u, w_i), (v, w_i)\}$, i.e.,  $B_n$ is composed of $n$ triangles (copies of $K_3$) that only share a common edge.
\end{enumerate}

We refer to \cite{Bondy, Diestel} for more graph-theoretic terminology and notation.

\smallskip

Throughout the paper, all groups are finite. Let $G$ be a group, $\mathscr{S}(G)$ the set of all subgroups of $G$, and 
\[\mathscr{V}(G) =\{\, H \leq G \mid 1 < H < G \,\}\]
the collection of all nontrivial proper subgroups of $G$. The \emph{factorization graph} of $G$, denoted by $\Gamma(G)$, is the graph with vertex set $\mathscr{V}(G)$, where two distinct vertices $H, K$ are adjacent if and only if $HK = G$.

If $\Gamma(G)$ contains isolated vertices, we define the \emph{reduced factorization graph}
$\Gamma'(G)$ as the induced subgraph obtained from $\Gamma(G)$ by deleting all isolated vertices. Given a subset $X\subseteq \mathscr{V}(G)$, we denote by $\Gamma(G)[X]$ the subgraph of $\Gamma(G)$ induced by $X$. In particular, induced subgraphs of $\Gamma'(G)$ will be denoted analogously.

Finally, given a group $G$ and subgroups $H,K\leq G$, we will repeatedly use the fact that
\[|HK|=\frac{|H||K|}{|H\cap K|},\]
as well as basic properties of subgroup lattices of cyclic, quaternion, and dicyclic groups.

The cyclic group of order $n$ is denoted by $C_n$. If $C_n=\gen{g}$, then for every divisor $d$ of $n$ there exists exactly one subgroup $U_d$ of order $d$, which is generated by $g^{n/d}$. In addition, for the subgroup $\gen{g^d}$ holds $|G:\gen{g^d}| =d$. 

If $H$ and $K$ are subgroups of finite index of a group $G$ such that $|G: H|$ and $|G: K|$ are relatively prime, then $G=HK$. (See, for example, \cite[I.2.13]{Huppert} or \cite[1.3.11]{Robinson1}. If $G$ is a finite group, we denote by $\pi(G)$ the number of distinct prime divisors of the order of $G$.

\section{General observations on factorization graphs}

In this section, we collect a small number of general facts about the factorization graph $\Gamma(G)$
that will be used repeatedly in later sections. We emphasize that no new constructions are introduced
here; the purpose is to fix simple criteria that clarify adjacency in $\Gamma(G)$ and streamline the
proofs for specific families of groups.

\subsection{Adjacency criteria}

We begin with a well-known sufficient condition for adjacency in $\Gamma(G)$.

\begin{lemma}\label{coprime}
Let $G$ be a finite group and let $H,K\leq G$ be subgroups such that
\[\gcd(|G:H|,|G:K|)=1.\]
Then $HK = G$, and hence $H$ and $K$ are adjacent in $\Gamma(G)$.
\end{lemma}

\begin{proof}
The index $|G:HK|$ divides both $|G:H|$ and $|G:K|$. Since $\gcd(|G : H|, |G : K|) = 1$, it follows that $|G:HK|=1$, that is, $HK = G$.
\end{proof}

The next observation explains why cliques appear naturally in factorization graphs.

\begin{remark}\label{maximal}
The maximal subgroups of a finite group do not necessarily induce a clique in
$\Gamma(G)$. In general, two distinct maximal subgroups $M$ and $N$ need not satisfy $MN=G$. However, this property does hold in several important situations considered later in this paper.

\begin{enumerate}
\item If the maximal subgroups have pairwise coprime indices, then Lemma \ref{coprime} implies that they are pairwise adjacent.

\item If the maximal subgroups are normal, then for any two distinct maximal
subgroups $M$ and $N$, the product $MN$ is a normal subgroup properly containing $M$, and therefore $MN=G$.
\end{enumerate}
\end{remark}

\subsection{Computational construction}
For completeness, we briefly describe a direct algorithmic approach to constructing the factorization graph of a finite group from its subgroup lattice.

Given a finite group $G$, the algorithm proceeds as follows: first, compute the list of all nontrivial proper subgroups of $G$; these form the vertex set of $\Gamma(G)$. For each unordered pair of distinct subgroups $H, K$, compute the product set $HK$ and check whether $HK = G$. If so, add an edge between the corresponding vertices.

This procedure is intentionally naive. If $|\mathscr{S}(G)|$ denotes the number of subgroups
of $G$, then the algorithm requires $O(|\mathscr{S}(G)|^2)$ subgroup products, each of which
may involve up to $|G|^2$ group operations. Even with extremely loose bounds, this illustrates
that the direct computation of $\Gamma(G)$ quickly becomes infeasible for large groups.

Nevertheless, such a construction is useful for experimental purposes and for visualizing
factorization graphs of small or moderately sized groups. An explicit implementation in
\textsc{SageMath} is provided in the appendix.

\section{The graph for cyclic groups of square-free order}

In this section, we study the factorization graph of cyclic groups of square-free order. This class provides a natural first test case, since the subgroup lattice of a cyclic group is completely determined by the divisors of its order, while the square-free assumption eliminates isolated vertices and yields a connected factorization graph. Our goal is to describe the global structure of $\Gamma(G)$ and determine several classical invariants such as planarity, girth, diameter, and pancyclicity.

\begin{remark}\label{HKsubgroup}
Let $G$ be a finite group and let $H, K\leq G$. The product set $HK$ is a subgroup of $G$ if and only if $HK=KH$. In this case, one has
\[HK=\langle H\cup K\rangle=\langle H,K\rangle.\]
In particular, if $\varphi:G_1\to G_2$ is a group isomorphism, then the map
$H\mapsto \varphi(H)$ induces an isomorphism $\Gamma(G_1)\cong \Gamma(G_2)$.
\end{remark}

\begin{lemma}\label{iso-lat}
If $G_1$ and $G_2$ are finite abelian groups with isomorphic subgroup lattices, then $\Gamma(G_1)\cong \Gamma(G_2)$.
\end{lemma}

\begin{proof}
Let $f$ be an isomorphism between the subgroup lattices of $G_1$ and $G_2$. If $H,K\in V(\Gamma(G_1))$ satisfy $G_1=HK$, then $G_2=f(G_1)=f(HK)$. Since $G_1$ and $G_2$ are abelian, for any subgroups $H, K$ one has $HK = \langle H, K\rangle$ (the join in the subgroup lattice). Since $f$ preserves joins, $f(HK) = f(\langle H, K\rangle) = \langle f(H), f(K)\rangle = f(H)f(K)$, and hence $f(H)f(K)=G_2$. Therefore $f$ induces an isomorphism $\Gamma(G_1)\cong \Gamma(G_2)$.
\end{proof}

\begin{theorem}\label{isomorfia} 
Let $G\cong C_n$ and $H\cong C_m$ be cyclic groups of square-free order. Then $\Gamma(G)\cong \Gamma(H)$ 
if and only if $\omega(n) = \omega(m)$, where $\omega(t)$ denotes the number of distinct prime divisors of $t$.
\end{theorem}

\begin{proof}
Suppose $n=p_1\cdots p_r$ and $m=q_1\cdots q_s$, where the $p_i$ and $q_j$ are distinct primes.

Assume first that $r=s$. Since every proper subgroup of $C_n$ corresponds uniquely to a proper subset of $\{p_1,\dots,p_r\}$, and adjacency in $\Gamma(C_n)$ is determined by whether the corresponding subsets together contain all prime divisors of $n$, it follows that the graph structure depends only on $r$. Hence $\Gamma(C_n) \cong \Gamma(C_m)$. 

Conversely, assume that $\Gamma(C_n) \cong \Gamma(C_m)$. Since graph isomorphisms preserve vertex degrees, they preserve the number of vertices of degree $1$. In $\Gamma(C_n)$, the vertices of degree $1$
are precisely the minimal subgroups (of prime order) of $C_n$, and there are exactly $r$
such subgroups. Similarly, $\Gamma(C_m)$ has exactly $s$ vertices of
degree $1$. Hence $r=s$, that is, $\omega(n)=\omega(m)$.
\end{proof}

\begin{remark}
Let $n$ be a positive integer. The radical $\rad(n)$ of $n$ is defined as the product of the distinct prime numbers dividing $n$. It is well known that the Frattini subgroup of the cyclic group $C_n$ of order $n$ is given by $\Phi(C_n) = C_{n/\rad(n)}$. In particular,  $\Phi(C_{p^a}) = C_{p^a-1}$, for a prime number $p$.
\end{remark}

\begin{lemma}\label{isolatednilpotent}
Let $G$ be a finite nilpotent group and let $p_1,\dots,p_r$ be the distinct prime divisors of $|G|$.
For each $j$, let $P_j\in \operatorname{Syl}_{p_j}(G)$ and set $a_j := |\mathscr{S}(\Phi(P_j))|$. Then the number of isolated vertices of $\Gamma(G)$ satisfies
\[|\mathrm{Iso}(\Gamma(G))| \geq \Bigg(\prod_{j=1}^r a_j\Bigg)-1.\]
Moreover, if $G$ is cyclic, then equality holds, and if $|G|$ is square-free, then $\Gamma(G)$ is connected.
\end{lemma}

\begin{proof} 
By \cite[Theorem 2.2.]{Das}, $\mathscr{S}(\Phi(G))-\{1\}$ is the collection of isolated vertices of  $\Gamma(G)$. Thus, since $\Phi(G)\cong \prod_{j=1}^{r}\Phi(P_{j})$, then $ |\mathscr{S}(\Phi(G))-\{1\}|\geq  (\prod_{j=1}^r a_{j}) -1$. In particular, if $G$ is cyclic, $\Phi(G)\cong C_{n/\rad(n)}$ and so $ |\mathscr{S}(C_{n/\rad(n)})-\{1\}|=$ $(\prod_{j=1}^r a_j) -1$. The rest is clear.
\end{proof}

For the remainder of this section, let $G=\langle g\rangle\cong C_{p_1\cdots p_r}$ be cyclic of square-free order. For each $1\leq j\leq r$, set $M_j:=\langle g^{p_j}\rangle$. Note that $M_1,\ldots, M_r$ are precisely the maximal subgroups of $G$. Since their indices are the distinct primes $p_1,\ldots,p_r$, Lemma \ref{coprime} implies that $M_i M_j=G$ whenever $i\neq j$. Hence $M=\{M_1,\ldots, M_r\}$ induces a clique in $\Gamma(G)$.

\begin{example}\label{Cpq}
The factorization graph of a cyclic group of non-square-free order is disconnected; indeed, it contains isolated vertices \cite[Theorem 2.2]{Das}. 
\end{example}

\begin{remark}\label{rge3}
If $r=1$, then $G\cong C_p$ has no nontrivial proper subgroups, and hence $\Gamma(G)$ has no vertices. If $r=2$, then $G\cong C_{pq}$ has exactly two nontrivial proper subgroups, and therefore $\Gamma(G)\cong K_2$. If $r=3$, then $G\cong C_{pqr}$ and $\Gamma(G)$ is the sunlet graph $L_6$. The following Figures show these two graphs.
 
\begin{figure}[H]
\label{F4}
\begin{center}
		\begin{tikzpicture}[mynode/.style={font=\color{#1}\sffamily,circle,inner sep=2pt,minimum size=0.2cm,draw,thick}] 
		\node[mynode=yellow,fill=] (q) at (-5,0) {}; 
		\node[mynode=yellow,fill=] (p) at (-5,2) {};
		\draw[thick] (p) -- (q)	;
		\node [below] at (-5,0) {$\langle g^{p}\rangle$}; 
		\node [above] at (-5,2) {$\langle g^{q}\rangle$}; 
			\node[mynode=yellow,fill=] (r) at (-1,1) {}; 
			\node[mynode=yellow,fill=] (p) at (-2,1) {};
			\node[mynode=yellow,fill=] (q) at  (-1.5,0)  {}; 
			\node[mynode=yellow,fill=] (pq) at (-0.5,2) {}; 
			\node[mynode=yellow,fill=] (qr) at (-2.5,2) {};
			\node[mynode=yellow,fill=] (pr) at  (-1.5,-1)  {};  
			\draw[thick] (pq) -- (r) -- (q) -- (p) -- (r)  (pr) -- (q) (qr) -- (p)	;
			\node [left] at (-2,1) {$\langle g^{p}\rangle$}; 
			\node [above] at (-2.5,2) {$\langle g^{qr}\rangle$}; 
			\node [right] at (-1,1) {$\langle g^{q}\rangle$}; 
			\node [above] at (-0.5,2) {$\langle g^{pr}\rangle$}; 
			\node [below] at (-1.5,-1) {$\langle g^{pq}\rangle$}; 
			\node [left] at (-1.5,0) {$\langle g^{r}\rangle$};  
		\end{tikzpicture}
	\caption{Factorization graph of $C_{pq}$ and $C_{pqr}$}
	\end{center}
\end{figure}
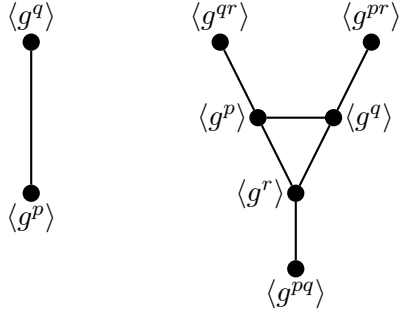
Thus, throughout the remainder of this section, we assume $r\geq 4$.
\end{remark}

The following theorem summarizes the main structural properties of the factorization graph
of cyclic groups of square-free order. The remainder of the section is devoted to its proof.

\begin{theorem}[Main theorem]\label{squarefree}
Let $G\cong C_{p_1p_2\cdots p_r}$ be cyclic of square-free order, where $p_1,\dots,p_r$ are distinct primes.

\begin{enumerate}
\item The graph $\Gamma(G)$ has order $|V(\Gamma(G))|=2^r-2$ and minimum degree $\delta(\Gamma(G))=1$.
Moreover, $\Delta(\Gamma(G))=2^{r-1}-1$. 

\item If $r\geq 4$, then $\Gamma(G)$ is non-planar.

\item The girth of $\Gamma(G)$ is $3$ for all $r\geq 3$, and $\Gamma(G)$ is claw-free if and only if $r\leq 3$.

\item The center of $\Gamma(G)$ is the clique induced by the maximal subgroups, and $\mathrm{rad}(\Gamma(G))=2$, while $\mathrm{diam}(\Gamma(G))=3$.

\item If $r \geq 4$, let $\Gamma'(G)$ be the induced subgraph obtained by deleting the vertices of degree $1$. Then $\Gamma'(G)$ is $2$-connected and contains cycles of every length from $3$ to $r$. Moreover, $\Gamma'(G)$ is Hamiltonian, with $C(\Gamma(G)) = 2^r - r - 2$, when $r = 4, 5$; this equality does not extend to $r \ge 6$.
\end{enumerate}
\end{theorem}

\begin{proof}
The proof is given throughout the section.
\end{proof}

\begin{remark}
The remainder of the section is devoted to proving Theorem \ref{squarefree}.
For convenience, the proofs are presented as a sequence of lemmas and corollaries.
\end{remark}

\begin{theorem}\label{planarity}
If $r\geq 4$, then $\Gamma(G)$ is non-planar.
\end{theorem}

\begin{proof}
Assume first that $r=4$, so $G\cong C_{p_1p_2p_3p_4}$. Consider the six subgroups
\[
H_i=\langle g^{p_i}\rangle\quad (i=1,2,3),
\qquad
K_j=\langle g^{p_4p_j}\rangle\quad (j=1,2,3).
\]
A direct verification shows that $H_iK_j=G$ for all $1\leq i,j\leq 3$.
Moreover, there are no edges among $\{H_1,H_2,H_3\}$, and similarly no edges among
$\{K_1,K_2,K_3\}$. Hence, the induced subgraph on these six vertices is isomorphic to $K_{3,3}$.
By Kuratowski's theorem, $\Gamma(G)$ is non-planar.

Now assume $r>4$. Choosing any four primes among $\{p_1,\dots,p_r\}$, the same construction yields
a subgraph isomorphic to $K_{3,3}$ inside $\Gamma(G)$. Therefore, $\Gamma(G)$ is non-planar for all
$r\geq 4$.
\end{proof}

\begin{lemma}\label{complemento-cyc}
Let $G=\langle g\rangle\cong C_{p_1\cdots p_r}$ be cyclic of square-free order.
\begin{enumerate}
\item If $H\leq G$ is a subgroup of order $|H| = p_{i_1}p_{i_2}\cdots p_{i_k}$,  with $1\leq k\leq r-1$, then there exists a subgroup $K\leq G$ of order
\[|K|= |G|/|H| = \frac{p_1\cdots p_r}{p_{i_1}\cdots p_{i_k}}\]
such that $HK = G$.

\item Let $A\in V(\Gamma(G))\setminus \mathcal{M}$ and write $A =\langle g^{p_{i_1}\cdots p_{i_k}}\rangle$ with $1<k<r$. Then $A$ is connected to exactly $r-k$ vertices of $\mathcal{M}$.
\end{enumerate}
\end{lemma}

\begin{proof}
(1) Let $K$ be the unique subgroup of $G$ of order $|G|/|H|$.
Since $|G:H|=|K|$ and $|G:K|=|H|$ are coprime (because $|G|$ is square-free), Lemma \ref{coprime}  yields $HK = G$.

(2) Let $A=\langle g^{p_{i_1}\cdots p_{i_k}}\rangle$ with $1<k<r$. If $p_j\notin \{p_{i_1},\dots,p_{i_k}\}$ and we set $B_j =\langle g^{p_j}\rangle\in \mathcal{M}$, then $|A| = |G|/p_{i_1}\cdots p_{i_k}$, and $|B_j| = |G|/p_j$. Moreover, since $p_j$ is not among the primes defining $A$, we have $A\cap B_j = \langle g^{p_{i_1}\cdots p_{i_k}p_j}\rangle$ and hence $|A\cap B_j|= |G|/p_{i_1}\cdots p_{i_k}p_j$. Hence
\[|AB_j| = \frac{|A||B_j|}{|A\cap B_j|} = |G|.\]
So $AB_j=G$ and therefore $A \sim B_j$ in $\Gamma(G)$. Since there are exactly $r-k$ primes not appearing in the product $p_{i_1}\cdots p_{i_k}$, the claim follows.
\end{proof}

\begin{proposition}\label{basic-invariants}
Let $G\cong C_{p_1\cdots p_r}$ be cyclic of square-free order.
\begin{enumerate}
\item $\delta(\Gamma(G))=1$.
\item $\Delta(\Gamma(G))=2^{r-1}-1$.
\item The degree sequence of $\Gamma(G)$ is
\[(\underbrace{2^{r-1}-1,\dots,2^{r-1}-1}_{\binom{r}{1}}, \underbrace{2^{r-2}-1,\dots,2^{r-2}-1}_{\binom{r}{2}}, \dots,
\underbrace{1,\dots,1}_{\binom{r}{r-1}}). \]
In particular, if $r\geq 3$ then $\Gamma(G)$ is a non-regular graph.
\item $\Gamma(G)$ has order $\sum_{j=1}^{r-1}\binom{r}{j}=2^r-2$
and size $\frac{1}{2}\sum_{j=1}^{r-1}\binom{r}{j}\,(2^{r-j}-1)$.
\item If $r\geq 3$, then $\Gamma(G)$ is neither Eulerian nor Hamiltonian.
\end{enumerate}
\end{proposition}

\begin{proof} 
(1) We define $k_j := \prod_{j\neq i=1}^r p_i$  for $j \in \{1,\ldots, r\}$ and consider the subgroup $A_j := \gen{g^{k_j}}$. A straightforward verification shows that $A_j$ is adjacent only to $\gen{g^{p_j}}$. Therefore, $\deg(A_j) = 1$.

\noindent $(2)$ Let $B_j :=\gen{g^{p_j}}$, then $B_j$ is adjacent to every subgroup of the form $\gen{g^{l_j}}$, where $l_j=p_i$ for all $i\neq j$ or $l_j$ is a product of primes of $\{p_1,\ldots, p_r\}\setminus \{p_j\}$. That is,
\begin{equation}\label{max-degree}
\deg(B_j) = \sum_{k=1}^{r-1}  \tbinom{r-1}{k}  = 2^{r-1}-1.
\end{equation}

\noindent $(3)$ For each $1\leq j\leq r-1$, there are exactly $\binom{r}{j}$ subgroups of $G$ whose order is a product of exactly $j$ distinct primes among $\{p_1,\dots,p_r\}$. By part (ii), each such subgroup has degree $2^{\,r-j}-1$. Hence the degree sequence of $\Gamma(G)$ is
\[(\underbrace{2^{r-1}-1,\dots,2^{r-1}-1}_{\binom{r}{1}}, \underbrace{2^{r-2}-1,\dots,2^{r-2}-1}_{\binom{r}{2}}, \dots,
\underbrace{1,\dots,1}_{\binom{r}{r-1}}), \]
as claimed. The non-regularity of $\Gamma(G)$ is a direct consequence of the degree sequence. 

\noindent $(4)$  Since every nontrivial proper subgroup of $G$ corresponds to a nonempty proper subset
of $\{p_1,\dots,p_r\}$, we obtain $|V(\Gamma(G))| = \sum_{j=1}^{r-1}\binom{r}{j} = 2^r-2$.
Moreover, by the Handshaking Theorem, $2|E(\Gamma(G))|= \sum_{j=1}^{r-1}\binom{r}{j}(2^{r-j}-1)$. 
\smallskip
\noindent $(5)$ Note that every vertex of $\Gamma(G)$ has odd degree, and therefore $\Gamma(G)$
cannot be Eulerian. Moreover, $\Gamma(G)$ contains a vertex of degree $1$, so it cannot be Hamiltonian.
\end{proof}

\begin{proposition}\label{smalls}
Let $G\cong C_{p_1\cdots p_r}$ be cyclic of square-free order, and assume $r\geq 3$.
\begin{enumerate}
\item $\Gamma(G)$ contains a unique subgraph isomorphic to $K_r\odot K_1$.
\item $\Gamma(G)$ contains a unique sunlet graph of order $2r$.
\item $g(\Gamma(G))=3$.
\item $\omega(\Gamma(G))=r$.
\end{enumerate}
\end{proposition}

\begin{proof}
$(1) - (2)$. Since the maximal subgroups have pairwise coprime indices, Lemma \ref{coprime} shows that $\Gamma(G)[M]\cong K_r$. For each $1\leq j\leq r$, consider the subgroup 
\[A_j:=\left\langle g^{\prod_{i\neq j}p_i}\right\rangle.\]
As observed in the proof of Proposition \ref{basic-invariants} (1), we have $\deg(A_j)=1$ and $A_j$ is adjacent only to $M_j$. Hence, the induced subgraph on the vertex set $\mathcal M\cup\{A_1,\dots,A_r\}$ is isomorphic to $K_r\odot K_1$.
\smallskip

To prove uniqueness, note that the vertices of degree $1$ in $\Gamma(G)$ are precisely the subgroups generated by products of $r-1$ primes. Each such vertex is adjacent to exactly one maximal subgroup,
so the clique $K_r$ determined by the maximal subgroups is the only possible candidate for the core of
a $K_r\odot K_1$ subgraph.
\smallskip

Finally, since $K_r$ contains a cycle of length $r$, attaching a pendant vertex to each vertex of that cycle
yields a sunlet graph of order $2r$.

\noindent $(3)$ Since $r\geq 3$, the induced subgraph $\Gamma(G)[\mathcal M]\cong K_r$ contains a triangle, hence $g(\Gamma(G))=3$.
\smallskip

\noindent $(4)$ Since $\mathcal M$ is a clique of size $r$, we have $\omega(\Gamma(G))\geq r$. On the other hand, any clique
must contain at most one subgroup for each prime divisor, so $\omega(\Gamma(G))=r$.
\end{proof}

\begin{proposition}\label{claw-free}
Let $G \cong C_{p_1 \cdots p_r}$ be cyclic of square-free 
order with $r \geq 2$. Then $\Gamma(G)$ is claw-free if and 
only if $r \leq 3$.
\end{proposition}

\begin{proof}
Suppose $r\leq 3$. By Remark \ref{rge3}, $\Gamma(G) \cong K_2$ for $r = 2$, and $\Gamma(G) \cong L_6$ for $r = 3$. Hence, $\Gamma(G)$ is claw-free for $r \leq 3$. 
\smallskip 

Conversely, assume $r \geq 4$. For each nonempty proper subset $S \subset \{1, \ldots, r\}$, denote by $H_S$ the unique subgroup of $G$ of order $\prod_{i \in S} p_i$. Note that two distinct vertices $H_S$ and $H_T$ are adjacent in $\Gamma(G)$ if and only if 
\begin{equation}\label{adjacents-4}
    S \cup T = \{1, \ldots, r\}.
\end{equation}
Consider the vertex $M_1 = H_{\{2,3,\ldots,r\}}$ and the three vertices $A:= H_{\{1,2\}}$, $B:= H_{\{1,3\}}$ and $C:= H_{\{1,4\}}$. By the adjacency criterion \eqref{adjacents-4}, each of these is adjacent to $M_1$. On the other hand, for any two of them, say $H_{\{1, i\}}$ and $H_{\{1, j\}}$ with $2 \leq i < j \leq 4$, we have that they are not adjacent. Hence, $\{M_1, A, B, C\}$ induces a $K_{1,3}$ with center $M_1$, and $\Gamma(G)$ is not claw-free.
\end{proof}

\begin{proposition}\label{prop:center-metric}
Let $G\cong C_{p_1\cdots p_r}$ be cyclic of square-free order, and assume $r\geq 3$.
\begin{enumerate}
\item The induced subgraph $\Gamma(G)[\mathcal{M}]\cong K_r$ is the center of $\Gamma(G)$.
\item $\mathrm{rad}(\Gamma(G))=2$ and $\mathrm{diam}(\Gamma(G))=3$.
\item The clique $\Gamma(G)[\mathcal{M}]$ contains the cycle $C_0 := (M_1,M_2,\dots,M_r,M_1)$, which we call the outer cycle of the center.
\end{enumerate}
\end{proposition}

\begin{proof}
$(1)$ We show that the vertices of minimum eccentricity are precisely the elements of $\mathcal{M}$. 

\noindent (i) Let $A\in \mathcal{M}$ and let $B$ be any other vertex with $A\neq B$. If $B\in \mathcal{M}$, then since $\Gamma(G)[\mathcal M]$ is a complete graph, we have $d(A,B)=1$. Assume now that $B\notin \mathcal{M}$. If $A \sim B$, then $d(A,B)=1$. If $A \not\sim B$ then there exists $C\in \mathcal{M}$ such that $C\sim B$. Therefore,  $d(A,B)=2$. This proves $\operatorname{ecc}(A)=2$.

\noindent (ii) Let $A\notin \mathcal{M}$ and let $B$ be any other vertex with $A\neq B$. If $B\in \mathcal{M}$, similar as above we have 
$d(A,B)=1$ or $d(A,B)=2$. Suppose now $B\notin \mathcal{M}$. Then using Lemma \ref{complemento-cyc} there exists $C, D\in \mathcal{M}$ such that $A\sim C$ and $B\sim D$. Since the maximal subgroups have pairwise coprime indices, Lemma \ref{coprime} implies that $\Gamma(G)[M]\cong K_r$, in particular, $C \sim D$. If $C=D$, then $d(A,B)=2$. In the other case, $d(A, B)=3$. Thus, $\operatorname{ecc}(A)=3$. We can conclude then that $\Gamma(G)[\mathcal M]$ is the center of $\Gamma(G)$, which is isomorphic to $K_r$

\smallskip

\noindent (2) Since vertices in $\mathcal{M}$ have eccentricity $2$, we obtain $\mathrm{rad}(\Gamma(G))=2$.
Moreover, vertices of degree $1$ exist by Proposition \ref{basic-invariants}, and they are at distance
$3$ from each other, so $\mathrm{diam}(\Gamma(G))=3$.

\smallskip

\noindent (3) This is immediate since $\Gamma(G)[M]\cong K_r$ contains a Hamiltonian cycle.
\end{proof}

\begin{lemma}\label{lem:completesun}
If $r\geq 4$, then $\Gamma(G)$ contains a unique complete sun graph on $2r$ vertices.
\end{lemma}

\begin{proof}
By Proposition \ref{prop:center-metric}, the outer cycle of the center is $C_0$. For each $1\leq j\leq r$, define
\[A_j:=\left\langle g^{\prod_{i\neq j,j+1} p_i}\right\rangle,\] 
where indices are taken modulo $r$. Clearly, $A_j\notin \mathcal{M}$ and $A_j$ is a proper nontrivial subgroup of $G$.

Moreover, by Lemma \ref{complemento-cyc} (2), the vertex $A_j$ is adjacent exactly to those maximal subgroups
$M_i$ such that $p_i$ does not appear in the product $\prod_{i\neq j, j+1}p_i$, namely $M_j$ and $M_{j+1}$.
Hence, each $A_j$ is adjacent precisely to two consecutive vertices of the outer cycle.

Therefore, the vertices $M_1,\dots, M_r,\ A_1,\dots, A_r$ span an induced subgraph in which $M_1,\dots,M_r$ form a clique and each $A_j$ is adjacent exactly to
$M_j$ and $M_{j+1}$. This induced subgraph is a complete sun graph.

Finally, uniqueness follows since the vertices of degree $2$ in $\Gamma(G)$ are precisely the subgroups generated
by products of exactly $r-2$ primes, and each such vertex is adjacent to exactly two consecutive vertices of the
outer cycle. Hence, no other complete sun subgraph can occur.
\end{proof}

\begin{theorem}\label{pancyclic}
Assume $r\geq 4$ and let $\Gamma'(G)$ be the induced subgraph of $\Gamma(G)$ obtained by deleting
the vertices of degree $1$. Then $\Gamma'(G)$ is $2$-connected.
\end{theorem}

\begin{proof}
$(1)$ We use the following standard characterization: a connected graph is $2$-connected if and only if every pair of distinct vertices lies on a common cycle. Let $A,B\in V(\Gamma'(G))$ with $A\neq B$. We show that $A$ and $B$ lie on a common cycle. First, note that $\Gamma(G)[\mathcal M]\cong K_r$, hence any two distinct vertices in $\mathcal M$
are adjacent.

\smallskip
\noindent\emph{Case 1:} $A,B\in \mathcal M$. Since $\Gamma(G)[\mathcal M]$ is complete and $r\geq 3$, there exists $C\in\mathcal M\setminus\{A,B\}$ and $ABCA$ is a triangle in $\Gamma'(G)$.

\smallskip
\noindent\emph{Case 2:} $A\in \mathcal M$ and $B\notin \mathcal M$. Since $B\in V(\Gamma'(G))$, it has degree at least $2$ in $\Gamma(G)$, and in particular $B$ has at least two neighbors in $\mathcal M$ (indeed, by Lemma \ref{complemento-cyc} (2) it is joined by an edge with exactly $r-k\geq 2$ vertices of $\mathcal M$ for some $k\leq r-2$). Choose two distinct vertices $C,D\in\mathcal M$ such that $B\sim C$ and $B\sim D$. Because $\mathcal M$ is a clique, we have $C\sim D$, and also $A\sim C$ and $A\sim D$. Hence $ACBDA$ is a cycle in $\Gamma'(G)$ containing $A$ and $B$.

\smallskip
\noindent\emph{Case 3: $A,B\notin \mathcal M$.} As above, choose distinct $C,D\in\mathcal M$ with $A\sim C$ and $A\sim D$, and distinct $E,F\in\mathcal M$ with $B\sim E$ and $B\sim F$. If $\{C,D\}\cap\{E,F\}\neq\emptyset$, take $X\in \{C,D\}\cap\{E,F\}$ and $Y\in\mathcal M$ with $Y\neq X$ and $A\sim Y$ (such a $Y$ exists since $A$ has at least two neighbors in $\mathcal M$).
Then $A, Y, X, B$ form a path. Since $\Gamma(G)[\mathcal M]\cong K_r$ is complete, there exists a path in $\mathcal M$ joining $X$ to $Y$. Hence $A$ and $B$ lie on a common cycle.
If $\{C,D\}\cap\{E,F\}=\emptyset$, then since $\mathcal M$ is complete we have edges $D\sim E$ and $F\sim C$. Therefore, $ADEBFCA$ is a cycle in $\Gamma'(G)$ containing $A$ and $B$.

In all cases, $A$ and $B$ belong to a common cycle. Hence $\Gamma'(G)$ is $2$-connected.
\end{proof}

\begin{proposition}
Let $r \ge 4$ and let $\Gamma'(G)$ be the induced subgraph of $\Gamma(G)$ obtained by deleting all vertices of degree $1$. Then $\Gamma'(G)$ contains a cycle of every length $\ell$ with $3 \le \ell \le r$.
\end{proposition}

\begin{proof}
By Remark \ref{maximal}, the maximal subgroups $M_1, \dots, M_r$ satisfy $M_iM_j = G$ for all $i \neq j$, so $\Gamma(G)[\mathcal{M}] \cong K_r$. Since each $M_i$ has degree $2^{r-1} - 1 \ge 3$ in $\Gamma(G)$ (Proposition \ref{basic-invariants}, no maximal subgroup is a pendant vertex, so $\mathcal{M} \subseteq V(\Gamma'(G))$. As $K_r$ contains cycles of every length from $3$ to $r$, so does the induced subgraph $\Gamma(G)[\mathcal{M}] \subseteq \Gamma'(G)$.
\end{proof}

\begin{remark}\label{rem:4.18-limits}
The stronger property of Hamiltonicity (and, more generally, pancyclicity) of $\Gamma'(G)$ does not extend to all $r \ge 4$. Exact computation shows that $C(\Gamma(G)) = 2^r - r - 2$ holds for $r = 4, 5$, but fails for $r \ge 6$:

\begin{center}
\begin{tabular}{c c c c}
\hline
$r$ & $2^r - r - 2$ & $C(\Gamma(G))$ & Hamiltonian \\
\hline
4 & 10  & 10  & yes \\
5 & 25  & 25  & yes \\
6 & 56  & 52  & no \\
7 & 119 & 108 & no \\
\hline
\end{tabular}
\end{center}

\smallskip

We leave open the determination of a closed-form expression for $C(\Gamma(G))$, and of the full spectrum of realizable cycle lengths in $\Gamma'(G)$, for $r \ge 6$.
\end{remark}

\section{The Graph for the generalized Quaternion group}

We now specialize to the generalized quaternion group $Q_{2^n}$ (of order $2^n$, with $n\geq 3$), which has a well-understood subgroup lattice. It is generated by elements $a$ and $b$ with presentation: 
\begin{equation}
Q_{2^n} = \langle a,b \mid a^{2^{n-1}} =1, b^2 = a^{2^{n-2}}, b^{-1}ab=a^{-1}\rangle. 
\end{equation}

\begin{example}\label{ex:smallQ}
The factorization graphs of $Q_8$ and $Q_{16}$ are displayed in Figures  \ref{FgG}. These small cases illustrate the general structure of $\Gamma(Q_{2^n})$: a central role is played by the subgroup $\langle a\rangle$, while the remaining vertices form two large symmetric families determined by the parity of the exponent in $a^kb$.
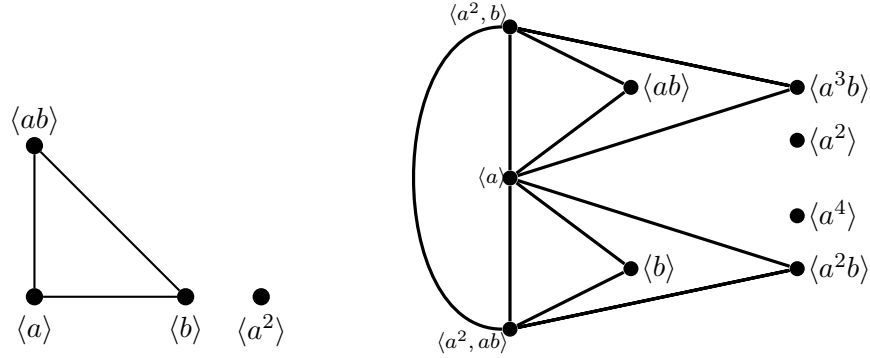
\begin{figure}[H]
\begin{center}
\begin{tikzpicture}[mynode/.style={font=\color{#1}\sffamily,circle,inner sep=2pt,minimum size=0.2cm,draw,thick}]  
			\node[mynode=yellow,fill=] (a) at (0,0) {};  
			\node[mynode=yellow,fill=] (b) at  (2,0)  {};   
			 \node[mynode=yellow,fill=] (ab) at (0,2) {}; 
			\node[mynode=yellow,fill=] (a2) at (3,0) {}; 
			 \draw[thick] (a)  -- (b) -- (ab) -- (a)  ;
			  \node [below] at (0,-0.1) {$\gen{a}$};   
			  \node [below] at (2,-0.1) {$\gen{b}$}; 
			 \node [below] at (3,-0.1) {$\gen{a^2}$}; 	 
			  \node [above] at (0,2) {$\gen{ab}$}; 	  
\end{tikzpicture} \hskip1cm
\begin{tikzpicture}[    vtx/.style={circle,fill=black,inner sep=2pt},
    lbl/.style={font=\scriptsize,inner sep=1pt},]

\def\k{3}           
\def\dx{1.6}        
\def\shift{1}     
\def\dy{1.2}        
\def\gap{2.0}       

\node[vtx] (Top) at (0,  \gap) {};
\node[lbl,above left=-1pt]  at (Top) {$\langle a^{2}, b\rangle$};

\node[vtx] (Mid) at (0,   0) {};
\node[lbl,left]             at (Mid) {$\langle a\rangle$};

\node[vtx] (Bot) at (0, -\gap) {};
\node[lbl,below left=-1pt]  at (Bot) {$\langle a^{2},ab\rangle$};

\draw[very thick] (Top) to[out=180,in=180] (Bot);   
\draw[very thick, bend left=15] (Top) -- (Mid) -- (Bot);          

\foreach \j in {1,...,\numexpr\k-2\relax}{
  \coordinate (Up\j) at (\j*\dx,  \dy);
  \node[vtx] (AU\j) at (Up\j) {};
  \draw[very thick, bend left=15] (Top) -- (AU\j) (Mid) -- (AU\j);
}


\coordinate (Upk) at ({\k*\dx-\shift}, \dy); 
\node[vtx] (AUk)  at (Upk) {};

\draw[very thick, bend left=15] (Top) -- (AUk) (Mid) -- (AUk) (Top) -- (AUk);    

\foreach \j in {1,...,\numexpr\k-2\relax}{
  \coordinate (Dw\j) at (\j*\dx, -\dy);
  \node[vtx] (AD\j) at (Dw\j) {};
  \draw[very thick, bend left=15] (Bot) -- (AD\j) (Mid) -- (AD\j);
}


\coordinate (Dwk) at ({\k*\dx-\shift}, -\dy); 
\node[vtx] (ADk)  at (Dwk) {};

\draw[very thick, bend left=15] (Bot) -- (ADk) (Mid) -- (ADk) (Bot) -- (ADk);    

\node [right] at  (1.6,1.2) {$\gen{ab}$}; 	  
\node [right] at  (3.8,1.2)  {$\gen{a^3b}$};
\node [right] at  (1.6,-1.2) {$\gen{b}$};
\node [right] at  (3.8,-1.2)  {$\gen{a^2b}$}; 	   	 
  
	 \node[circle,fill=, inner sep=2pt] (a2) at (3.8,0.5) {};
			  \node [right] at (3.8,0.5)  {$\gen{a^2}$}; 	 
	 \node[circle,fill=,inner sep=2pt] (a4) at (3.8,-0.5) {};	   	    
			  \node [right] at (3.8,-0.5) {$\gen{a^4}$}; 	 
\end{tikzpicture}
\caption{The factorization graph of $Q_{2^3}$ and $Q_{2^4}$}
	\label{FgG}
\end{center}
\end{figure} 
\end{example}

We will refer back to Figure \ref{FgG} as a guiding example throughout this section.

\begin{theorem}\label{girth_Q}
Let $n\geq 3$. Then the girth of $\Gamma(Q_{2^n})$ is 3.
\end{theorem}

\begin{proof}
The group $Q_{2^n}$ has exactly three maximal subgroups, namely $\langle a\rangle$, $\langle a^2,b\rangle$ and $\langle a^2,ab\rangle$. These subgroups are normal. Hence, for any two distinct maximal subgroups $M$ and $N$, the product $MN$ is a normal subgroup properly containing $M$. By maximality, it follows that $MN=Q_{2^n}$. Therefore, the three maximal subgroups induce a triangle in $\Gamma(Q_{2^n})$, and so $g(\Gamma(Q_{2^n}))=3$.
\end{proof}

In the following, let us consider $\Gamma'(Q_{2^n})$ as the subgraph obtained from $\Gamma(Q_{2^n})$ by removing the isolated vertices.

\begin{theorem}\label{puntosaislados}
Let $n\geq 3$. Then the isolated vertices of $\Gamma(Q_{2^n})$ are the subgroups $\langle a^{d}\rangle$, where $d\mid 2^{n-1}$ and $d\neq 1$. In particular, the number of isolated vertices is $n-2$.  
\end{theorem}

\begin{proof} 
By \cite[Theorem 5.8]{Conrad}, the vertex set of $\Gamma(Q_2^{n})$ is 
\[V = \{\langle a^{i}\rangle,\langle a^{j},a^kb\rangle  : i\mid2^{n-1},j\mid2^{n-2} \mbox{ and } 0\leq k<j\}\]
Let $A = \langle a^{i} \rangle$ with $i\mid 2^{n-1}$. We consider two cases:

\noindent \emph{Case 1.} Let $B=\langle a^l\rangle$ with $l\mid 2^{n-1}$.
If $i=1$, then $A=\langle a\rangle$ and clearly $AB\leq \langle a\rangle\neq Q_{2^n}$.
Assume now $i>1$. Without loss of generality, we may take $l<i$.
Then $AB=B$, so $|AB|<2^n$ and hence $A\not\sim B$.

\noindent \emph{Case 2.} Let $B=\langle a ^j, a^kb\rangle$. Then $A\cap B = \langle a^{\lcm(i,j)}\rangle$ and we have
\[|AB| = \frac{|A||B|}{|A\cap B|} = \frac{2^{n-1}/i \cdot 2^{n}/j}{2^{n-1}/\lcm(i,j)} = \frac{2^{n}}{\gcd(i,j)}< 2^n.\]
This proves that $A$ and $B$ are not connected. 
\end{proof}

\begin{proposition}\label{prop:girthQ}
Let $n\geq 3$. Every vertex of $\Gamma'(Q_{2^n})$ is connected by an edge to some element of $\mathcal{M}= \{\langle a^2, b\rangle, \langle a^2, ab\rangle, \langle a\rangle\}$.
\end{proposition}

\begin{proof}
Let $A$ be any vertex of $\Gamma'(Q_{2^n})$. Then there exists $B\in \mathcal{M}$ such that $A\not\leq B$. Since $B$ is a normal maximal subgroup of $Q_{2^n}$, it follows $AB=Q_{2^n}$. 
\end{proof}

\begin{proposition}\label{cor:totaldom}
Let $n\geq 3$. Then 
\begin{enumerate}
    \item $\{\langle a\rangle\}$ is a dominating set of $\Gamma'(Q_{2^n})$. In particular, $\gamma(\Gamma'(Q_{2^n}))=1$. 
    \item The center of $\Gamma'(Q_{2^n})$ is $\langle a\rangle$.
    \item $\mathrm{diam} (\Gamma'(Q_{2^n})) = 2$. 
\end{enumerate}
\end{proposition}

\begin{proof}
(1) Let $H\in V(\Gamma'(Q_{2^n}))$. Then $H$ is a non-isolated vertex, hence it is not contained in
$\langle a\rangle$. Since $\langle a\rangle$ is maximal, we obtain $H\langle a\rangle=Q_{2^n}$.
Thus every vertex of $\Gamma'(Q_{2^n})$ is adjacent to $\langle a\rangle$, and therefore
$\{\langle a\rangle\}$ is a dominating set of $\Gamma'(Q_{2^n})$.

$(2) - (3)$ Let $A$ and $B$ be two vertices of $\Gamma'(Q_{2^n})$ distinct from $\langle a\rangle$ with $A\neq B$. 
\begin{enumerate}[i)]
\item From (1) we have $d(\langle a\rangle,A)=1$ for every vertex $A\neq \langle a\rangle$.
Hence $\mathrm{ecc}(\langle a\rangle)=1$.

\item If $A, B\in \mathcal{M}$, then by Proposition \ref{prop:girthQ} we have $d(A,B) = 1$.
On the other hand, if $A\in \mathcal{M}$ and $B\notin \mathcal{M}$, then $d(A,B)=1$ or $d(A,B)=2$. Therefore, $\operatorname{ecc}(A) = 2$. 

\item  If $A, B\notin \mathcal{M}$, then by (1) we have $d(A,B) = 2$.
\end{enumerate}
From the above, it follows that the center of $\Gamma'(Q_{2^n})$ is $\langle a\rangle$ and $\mathrm{diam}(\Gamma'(Q_{2^n})) = 2$.
\end{proof}


\begin{theorem}\label{estar_extension}
Let $n\geq 3$. The graph $\Gamma'(Q_{2^n})$ has a maximal induced star $S$ with central vertex $\langle a\rangle$ and $2(2^{n-2}-1)$ leaves. In particular, $S$ is a spanning tree of $\Gamma'(Q_{2^n})$.
\end{theorem}

\begin{proof}

We first describe the vertices outside the maximal subgroups, then determine their number, and finally show that they form a star centered at $\langle a\rangle$.

Let $\mathcal{N}$ be the set of vertices of $\Gamma'(Q_{2^{n}})$ excluding the set $\mathcal{M}$ of all maximal subgroups. Every subgroup in $\mathcal{N}$ has the form
\[B_{j,k}=\langle a^{2^{\,n-j}},a^k b\rangle,\]
where $2\leq j\leq n-2$ and $0\leq k<2^{n-j}$.

\smallskip

For fixed $j$, the subgroups $B_{j,k}$ are pairwise distinct.
Indeed, if $B_{j,k} = B_{j,k'}$, then $a^{k-k'}\in B_{j,k}\cap \langle a\rangle =\langle a^{2^{n-j}}\rangle$, which implies $2^{n-j}\mid (k-k')$.
Since $0\leq k,k'<2^{n-j}$, we obtain $k=k'$. Hence, for each $j$ there are exactly $2^{n-j}$ distinct subgroups of this form. Therefore,
\[|\mathcal{N}| =\sum_{j=2}^{n-2}2^{n-j} =2^{n-1}-4.\]

\smallskip

We now show that every vertex in $\mathcal{N}$ is adjacent only to $\langle a\rangle$. Let $B_{j,k}\in \mathcal{N}$. Since $B_{j,k}\not\leq \langle a\rangle$ and $\langle a\rangle$ is maximal, we have $\langle a\rangle B_{j,k}=Q_{2^n}$, so $B_{j,k}\sim \langle a\rangle$.

On the other hand, if $B_{j,k},B_{j',k'}\in \mathcal{N}$, then $B_{j,k} B_{j',k'}\neq Q_{2^n}$, since both are contained in one of the maximal subgroups $\langle a^2,b\rangle$ or $\langle a^2,ab\rangle$. Thus, no two vertices of $\mathcal{N}$ are adjacent.

\smallskip

\noindent\textbf{Step 4. The star structure.} The induced subgraph on $\mathcal{N} \cup\{\langle a\rangle\}$ 
is a star centered at $\langle a\rangle$ with $|\mathcal{N}|=2^{n-1}-4$ leaves. Finally, adding the two maximal subgroups $\langle a^2,b\rangle$ and $\langle a^2,ab\rangle$ gives a star with $2(2^{n-2}-1)$ leaves. Since every vertex of $\Gamma'(Q_{2^n})$ belongs to this subgraph, it is a spanning tree. 
\end{proof}

\begin{corollary}\label{cor:orderQ}
Let $n\geq 3$.
\begin{enumerate}
\item The order of $\Gamma'(Q_{2^n})$ is $2^{n-1}-1$.
\item The order of $\Gamma(Q_{2^n})$ is $2^{n-1}+n-3$.
\end{enumerate}
\end{corollary}

\begin{proof}
By Theorem \ref{estar_extension}, the subgraph $S$ is a star centered at $\langle a\rangle$.
Moreover, $\deg(\langle a\rangle)=2(2^{n-2}-1)$.
Since $S$ is a spanning tree, it contains all vertices of $\Gamma'(Q_{2^n})$, and therefore
\[|V(\Gamma'(Q_{2^n}))| = 1+\deg(\langle a\rangle) = 1+2(2^{n-2}-1) = 2^{n-1}-1.\]
Finally, by Proposition \ref{prop:girthQ}, the isolated vertices of $\Gamma(Q_{2^n})$ are precisely the cyclic subgroups
$\langle a^{2^t}\rangle$ for $2\leq t\leq n-1$, which gives $n-2$ additional vertices. Hence
\[|V(\Gamma(Q_{2^n}))| = |V(\Gamma'(Q_{2^n}))|+(n-2) = 2^{n-1}+n-3.\]
\end{proof}

\begin{lemma}\label{lem:twohalves}
Let $n\geq 4$. Define
\[I=\{\langle a^j,a^kb\rangle \mid j\mid 2^{n-2},\ j>2,\ 0\leq k<j,\ k \text{ even}\},\]
\[J=\{\langle a^j,a^kb\rangle \mid j\mid 2^{n-2},\ j>2,\ 0\leq k<j,\ k \text{ odd}\}.\]
Then $|I|=|J|=2(2^{n-3}-1)$. Moreover, the induced subgraph on $I\cup\{\langle a^2,ab\rangle\}$ is a star centered at
$\langle a^2,ab\rangle$, and the induced subgraph on $J\cup\{\langle a^2,b\rangle\}$ is a star centered at
$\langle a^2,b\rangle$. Each of these stars has $2(2^{n-3}-1)$ leaves.
\end{lemma}

\begin{proof}
For each divisor $j\mid 2^{n-2}$ with $j>2$ and each $0\leq k<j$,  we set $H_{j,k} := \langle a^j, a^k b\rangle$.

\smallskip
\noindent\emph{Counting.}
Fix $j$ as above. If $H_{j,k}=H_{j,k'}$, then $a^{k-k'}=(a^k b)(a^{k'}b)^{-1}\in H_{j,k}\cap \langle a\rangle$. Since $H_{j,k}\cap\langle a\rangle=\langle a^j\rangle$, we get $a^{k-k'}\in\langle a^j\rangle$, hence
$j\mid (k-k')$. As $0\leq k,k'<j$, this forces $k=k'$. Therefore, for each fixed $j$ there are exactly $j$
distinct subgroups $H_{j,k}$, and exactly $j/2$ of them have $k$ even (resp.\ odd). Consequently,
\[|I|=|J|=\sum_{\substack{j\mid 2^{n-2}\\ j>2}}\frac{j}{2} =\sum_{t=2}^{n-2}2^{t-1} =2^{n-2}-2=2(2^{n-3}-1).\]

\noindent\emph{The induced subgraph on $I\cup\{\langle a^2,ab\rangle\}$ is a star.}
Let $C_{\mathrm{odd}}:=\langle a^2,ab\rangle$ and $C_{\mathrm{even}}:=\langle a^2,b\rangle$.
If $H_{j,k}\in I$, then $k$ is even, so $a^k b$ has the form $a^{2t}b$. Hence, $H_{j,k}\leq \langle a^2,b\rangle=C_{\mathrm{even}}$, and in particular, for any two distinct vertices $H,H'\in I$ we have $HH'\leq C_{\mathrm{even}}\neq Q_{2^n}$,
so $H\not\sim H'$.

On the other hand, $H_{j,k}\not\leq C_{\mathrm{odd}}$, because every element of $C_{\mathrm{odd}}\setminus\langle a\rangle$
has the form $a^{2t+1}b$ (odd exponent), whereas $H_{j,k}$ contains $a^{2t}b$ (even exponent).
Since $C_{\mathrm{odd}}$ is maximal, $H_{j,k}C_{\mathrm{odd}}=Q_{2^n}$, i.e.\ $H_{j,k}\sim C_{\mathrm{odd}}$.
Thus, in the induced subgraph on $I\cup\{C_{\mathrm{odd}}\}$ the vertex $C_{\mathrm{odd}}$ is adjacent to every
vertex of $I$, and there are no edges among vertices of $I$, so it is a star centered at $C_{\mathrm{odd}}$.

\smallskip
The argument for the star on $J\cup\{\langle a^2,b\rangle\}$ is symmetric. 
Finally, each star has $|I|=|J|=2(2^{n-3}-1)$ leaves.
\end{proof}

\begin{theorem}\label{book2}
Let $n\geq 4$. In $\Gamma(Q_{2^n})$ there are two maximal triangular book subgraphs $B_p$ with $p=2(2^{n-3}-1)$.   
\end{theorem}

\begin{proof}
Let $\mathcal{N}$ be the set of vertices of $\Gamma'(Q_{2^{n}})$ excluding the set $\mathcal{M}$ of all maximal subgroups of $Q_{2^{n}}$. Then
\[\mathcal{N} = \{\langle a^{j},a^kb\rangle\mid j\mid2^{n-2},\ j>2 \mbox{ and } 0\leq k<j\} = \mathcal{I}\cup \mathcal{J},\]
where $\mathcal{I} = \{\langle a^{j},a^kb\rangle\mid  0\leq k<j \mbox{ and } k  \mbox{ even}\}$, and $\mathcal{J} = \{\langle a^{j},a^kb\rangle\mid  0\leq k<j \mbox{ and } k  \mbox{ odd}\}$. 
Thus, by Theorems \ref{estar_extension} and \ref{lem:twohalves}, every vertex in $I$ yields a triangle
$\langle a\rangle \sim \langle a^2,ab\rangle \sim \langle a^j,a^kb\rangle \sim \langle a\rangle$,
and similarly, every vertex in $J$ yields a triangle based on the edge
$\langle a\rangle \sim \langle a^2,b\rangle$.
Hence, in $\Gamma'(Q_{2^n})$ there are two triangular book subgraphs, each with base edges $\langle a^2, ab\rangle\sim \langle a\rangle$  and $\langle a^2,b\rangle\sim \langle a\rangle$. Moreover, since $|\mathcal{I}| = |\mathcal{J}|=2(2^{n-3}-1)$, the number of triangles in each book is $p = 2(2^{n-3}-1)$.
\end{proof}

\begin{proposition}\label{invariants-quaternion}
Let $n\geq 4$. Then the reduced factorization graph $\Gamma'(Q_{2^n})$ satisfies:
\begin{enumerate}
\item $\omega(\Gamma'(Q_{2^n}))=3$.
\item $\alpha(\Gamma'(Q_{2^n})) = 4(2^{n-3}-1)$.
\item $|E(\Gamma'(Q_{2^n}))| = 2(2^{n-1}-3)$.
\item $\Gamma'(Q_{2^n})$ is Eulerian.
\item $\Gamma'(Q_{2^n})$ is not Hamiltonian.
\end{enumerate}
\end{proposition}

\begin{proof}
(1) The three maximal subgroups $\langle a\rangle$, $\langle a^2,b\rangle$ and $\langle a^2,ab\rangle$
form a triangle in $\Gamma'(Q_{2^n})$, so $\omega(\Gamma'(Q_{2^n}))\geq 3$.
On the other hand, $Q_{2^n}$ has exactly three maximal subgroups, hence no clique of size $4$ can occur.
Therefore $\omega(\Gamma'(Q_{2^n}))=3$.

\smallskip
(2) By Lemma \ref{lem:twohalves}, the sets $I$ and $J$ are independent in $\Gamma'(Q_{2^n})$ and satisfy
$|I|=|J|=2(2^{n-3}-1)$. Hence $I\cup J$ is an independent set of size $4(2^{n-3}-1)$, and therefore
\[\alpha(\Gamma'(Q_{2^n}))\geq 4(2^{n-3}-1).\]
Conversely, every vertex of $\Gamma'(Q_{2^n})$ outside $\{\langle a\rangle,\langle a^2,b\rangle,\langle a^2,ab\rangle\}$
lies in $I\cup J$. Since $\langle a\rangle$ is adjacent to every vertex of $\Gamma'(Q_{2^n})$,
no independent set can contain $\langle a\rangle$ together with a vertex of $I\cup J$.
Moreover, any independent set can contain at most one of the vertices $\langle a^2,b\rangle$ and $\langle a^2,ab\rangle$,
since they are adjacent.
Thus, the maximum possible size is attained by $I\cup J$, and
\[\alpha(\Gamma'(Q_{2^n})) = 4(2^{n-3}-1).\]

\smallskip
(3) By Theorem \ref{estar_extension}, the vertex $\langle a\rangle$ has degree $2(2^{n-2}-1)$ in $\Gamma'(Q_{2^n})$.
Moreover, by Lemma \ref{lem:twohalves} the vertices in $I\cup J$ have degree $1$ inside the stars induced by
$I\cup\{\langle a^2,ab\rangle\}$ and $J\cup\{\langle a^2,b\rangle\}$.
Finally, $\langle a^2,b\rangle$ and $\langle a^2,ab\rangle$ are adjacent to $\langle a\rangle$ and to all vertices
in their corresponding half, so each of them has degree $2(2^{n-3}-1)+1$.

Applying the Handshaking Theorem gives
\[2|E(\Gamma'(Q_{2^n}))| = 2(2^{n-2}-1) +2\bigl(2(2^{n-3}-1)+1\bigr) +2\cdot 2(2^{n-3}-1), \]
which simplifies to $|E(\Gamma'(Q_{2^n}))| = 2(2^{n-1}-3)$.
\smallskip

(4) By the explicit degree computation above, every vertex of $\Gamma'(Q_{2^n})$ has even degree.
Hence $\Gamma'(Q_{2^n})$ is Eulerian.

\smallskip
(5) Since $\Gamma'(Q_{2^n})$ contains vertices of degree $1$ (namely those in $I\cup J$), it cannot be Hamiltonian.
\end{proof}

\begin{corollary} \label{sequenceGQ}   
 The degree sequence of $\Gamma(Q_{2^{n}})$ for $n\ge4$  is given by
\begin{equation}
  \bigl(      2^{\,n-1}-2,\;
      2^{\,n-2},\;2^{\,n-2},\;
      \underbrace{2,\dots,2}_{\;2^{\,n-1}-4},\;
      \underbrace{0,\dots,0}_{\,n-2}  \bigr)
\end{equation}
\end{corollary}

\begin{proof}
It follows immediately from Theorems  \ref{puntosaislados}, \ref{estar_extension}, and Proposition \ref{invariants-quaternion}.
\end{proof}

\begin{remark}
The case $n=3$ does not fall under Corollary \ref{sequenceGQ}, but the degree sequence of $\Gamma(Q_8)$
can be computed directly from Figure 2.
\end{remark}

\begin{corollary}\label{tamañodeQ}
Let $n\geq 3$. The number of triangles in $\Gamma'(Q_{2^n})$ is $2^{n-1}-3$. 
\end{corollary}

\begin{proof}
Using Theorems  \ref{book2}, \ref{girth_Q} and item (2) of Proposition \ref{invariants-quaternion} we have that the number of triangles in $\Gamma'(Q_{2^n})$ is $2^{n-1}-3$. 
\end{proof}

We now turn to the planarity of the reduced factorization graph $\Gamma'(Q_{2^n})$.
The cases $n=3$ and $n=4$ can be verified directly from Figure \ref{FgG}. For $n\geq 5$, the graph admits a particularly rigid structure: it is obtained by gluing two large triangular book subgraphs along the common vertex $\langle a\rangle$. This structural decomposition provides a direct proof of planarity.

\begin{theorem}\label{planarQ}
$\Gamma'(Q_{2^n})$ is a planar graph for $n\geq 3$.
\end{theorem}

\begin{proof}
Assume now $n \geq 5$. By Theorem \ref{book2}, the graph $\Gamma'(Q_{2^n})$ contains two maximal triangular book subgraphs, each isomorphic to $B_p$, where $p = 2(2^{n-3}-1)$. Each triangular book $B_p$ is planar. Moreover, these two copies intersect exactly in the vertex $\langle a \rangle$. Hence, $\Gamma'(Q_{2^n})$ is the one-point union of two planar graphs. Since the one-point union of planar graphs is planar, it follows that $\Gamma'(Q_{2^n})$ is planar.
\end{proof}

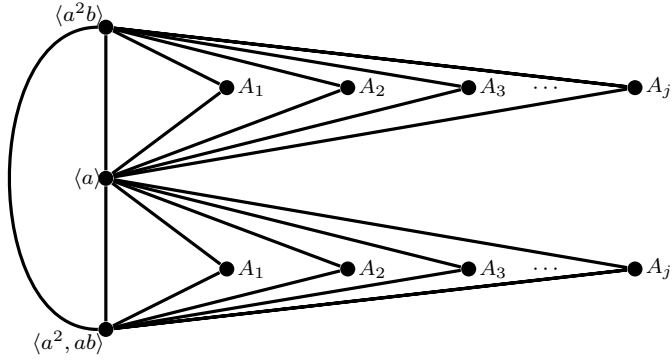
\begin{figure}[H]
  \centering
\begin{tikzpicture}[    vtx/.style={circle,fill=black,inner sep=2pt},
    lbl/.style={font=\scriptsize,inner sep=1pt},]

\def\k{5}           
\def\dx{1.6}        
\def\shift{1}     
\def\dy{1.2}        
\def\gap{2.0}       

\node[vtx] (Top) at (0,  \gap) {};
\node[lbl,above left=-1pt]  at (Top) {$\langle a^{2}b\rangle$};

\node[vtx] (Mid) at (0,   0) {};
\node[lbl,left]             at (Mid) {$\langle a\rangle$};

\node[vtx] (Bot) at (0, -\gap) {};
\node[lbl,below left=-1pt]  at (Bot) {$\langle a^{2},ab\rangle$};

\draw[very thick] (Top) to[out=180,in=180] (Bot);   
\draw[very thick, bend left=15] (Top) -- (Mid) -- (Bot);          

\foreach \j in {1,...,\numexpr\k-2\relax}{
  \coordinate (Up\j) at (\j*\dx,  \dy);
  \node[vtx] (AU\j) at (Up\j) {};
  \node[lbl,anchor=west] at (Up\j) {\ $A_{\j}$};
  \draw[very thick, bend left=15] (Top) -- (AU\j) (Mid) -- (AU\j);
}

\coordinate (dotsU) at ({(\k-1.5)*\dx}, \dy);
\node[lbl,anchor=west] at (dotsU) {$\cdots$};

\coordinate (Upk) at ({\k*\dx-\shift}, \dy); 
\node[vtx] (AUk)  at (Upk) {};
\node[lbl,anchor=west] at (Upk) {\ $A_{j}$};

\draw[very thick, bend left=15] (Top) -- (AUk) (Mid) -- (AUk) (Top) -- (AUk);    

\foreach \j in {1,...,\numexpr\k-2\relax}{
  \coordinate (Dw\j) at (\j*\dx, -\dy);
  \node[vtx] (AD\j) at (Dw\j) {};
  \node[lbl,anchor=west] at (Dw\j) {\ $A_{\j}$};
  \draw[very thick, bend left=15] (Bot) -- (AD\j) (Mid) -- (AD\j);
}

\coordinate (dotsD) at ({(\k-1.5)*\dx}, -\dy);
\node[lbl,anchor=west] at (dotsD) {$\cdots$};

\coordinate (Dwk) at ({\k*\dx-\shift}, -\dy); 
\node[vtx] (ADk)  at (Dwk) {};
\node[lbl,anchor=west] at (Dwk) {\ $A_{j}$};

\draw[very thick, bend left=15] (Bot) -- (ADk) (Mid) -- (ADk) (Bot) -- (ADk);    
\end{tikzpicture}
 
  \caption{A planar representation of $\Gamma'(Q_{2^n})$}
  \label{figura y}
\end{figure}

\begin{theorem}\label{thm:chromaticQ}
For $n\geq 4$, we have $\chi(\Gamma'(Q_{2^n}))=3$.
\end{theorem}

\begin{proof}
By Proposition \ref{invariants-quaternion} (1), the graph $\Gamma'(Q_{2^n})$ contains a triangle.
Hence $\chi(\Gamma'(Q_{2^n}))\geq 3$.

We now construct a proper $3$-coloring. Assign color $1$ to the vertex $\langle a\rangle$.
By Lemma \ref{lem:twohalves}, every vertex in $I$ is adjacent to $\langle a^2,ab\rangle$ and every vertex in $J$
is adjacent to $\langle a^2,b\rangle$. Color all vertices in $I\cup J$ with color $2$.
Finally, assign color $3$ to the vertices $\langle a^2,b\rangle$ and $\langle a^2,ab\rangle$.

This is a proper coloring since:
\begin{itemize}
\item $\langle a\rangle$ is adjacent to every vertex of $\Gamma'(Q_{2^n})$;
\item there are no edges inside $I\cup J$ by Lemma \ref{lem:twohalves};
\item the vertices $\langle a^2,b\rangle$ and $\langle a^2,ab\rangle$ are adjacent to $\langle a\rangle$ and to their
respective halves.
\end{itemize}
Therefore $\chi(\Gamma'(Q_{2^n}))\leq 3$, and the result follows.
\end{proof}

\section{The Graph for the Dicyclic group}

Let $m\geq 3$ be an odd integer. The \emph{Dicyclic group} $Q_{4m}$ has order $4m$ and the following presentation
\[Q_{4m} = \bigl\langle a,b \mid a^{2m}=1, \, b^{2}=a^{m},\; bab^{-1} = a^{-1} \bigr\rangle .\]

The set of all non–trivial subgroups of $Q_{4m}$, denoted by $\mathcal{S}$, is 
\begin{equation}\label{S}
\mathcal S = \bigcup_{\substack{d\, \mid \, 2m\\[2pt] 1\leq d < 2m}}
\bigl\langle a^{\,d}\bigr\rangle \;\cup\; \bigcup_{k=0}^{m-1}
        \bigl\langle a^{k} b \bigr\rangle.
\end{equation}

\smallskip

It is well known that Burnside classified 2-groups with a unique involution: they are either cyclic or generalized quaternion groups. See \cite[Chapter V, Theorem V]{Burnside}.

Let $n$ be a natural number. The divisor sum function of n, denoted by $\sigma (n)$, is defined
\[{\displaystyle \sigma (n)=\sum _{d\, \mid\, n} d}.\]
The divisor count function of $n$, denoted $\tau (n)$, is defined as the number of positive divisors of $n$. 

\begin{theorem}
The size of the set $S$ defined in \eqref{S} is $|\mathcal S| = 2\tau(m) + m - 1$.
\end{theorem}

\begin{proof}
Every proper nontrivial subgroup of $Q_{4m}$ falls into exactly one of the following two classes, which are disjoint.
\begin{enumerate}[i)]
\item \emph{Subgroups contained in $\langle a\rangle$.} The subgroup $\langle a\rangle\cong C_{2m}$ is cyclic, so it     possesses $\tau(2m)-1$ proper subgroups. Since $\gcd(2,m)=1$ and $\tau(2m) = \tau(2)\,\tau(m)=2\,\tau(m)$. Hence, the number of nontrivial subgroups of $\langle a\rangle$ is $2\tau(m)-1$. 

\item \emph{Subgroups generated by an element outside $\langle a\rangle$.} Let $x=a^{k}b\notin\langle a\rangle$.  Then $x^{2}=a^{m}$ and $x^{4}=1$; thus $K_j:=\langle x\rangle$ is cyclic of order $4$. The cosets $a^{k}b\langle a\rangle$ with $0\leq k<m$ are all distinct, giving exactly $m$ subgroups $K_0,\dots,K_{m-1}$. Any subgroup containing an element of the form $a^{k}b$ equals $\langle a^{k}b\rangle$, so no further subgroups occur.
\end{enumerate}
Summing both classes gives $|\mathcal S| =2\tau(m)+ m-1$.
\end{proof}

In the next theorem we can observe that $Q_{4m}$ only admits factorizations of the form $Q_{4m}= H\langle a^{k}b\rangle$ with $H=\langle a\rangle$ or $H=\langle a^{2}\rangle$ and $0\leq k\leq m-1$.


The next theorem provides a complete description of the reduced factorization graph in this section. All subsequent results follow as direct consequences.

\begin{theorem}\label{subq4}
For non–trivial subgroups $H$ and $K$ of $Q_{4m}$ we have $Q_{4m}= HK$ if and only if $H=\langle a\rangle$ or $H=\langle a^{2}\rangle$ and $K= \langle a^{k}b\rangle$ with $0\leq k\leq m-1$.
\end{theorem}

\begin{proof}
Suppose that $Q_{4m} = HK$ with $H,K<G$. At least one of the subgroups $H, K$ contains an element outside $\langle a\rangle$. Otherwise $HK \subseteq \langle a\rangle \neq G$. By \eqref{S}, every subgroup not contained in $\langle a\rangle$ has the form $\langle a^k b\rangle$.
Hence, without loss of generality, suppose $K=\langle a^{k}b\rangle$ for some $k$  and $H=\langle a^{d}\rangle \leq \langle a\rangle$ with $d\mid 2m$ and $d<2m$. Then $|H|=2m/d$. 
\begin{itemize}
\item If $d\mid m$, then $a^{m}\in H$, so $|H\cap K|=2$. Hence , $|HK| = \frac{4(2m/d)}{2}=4m/d$.
For this to equal $4m$, we need $d=1$, that is $H=\langle a\rangle$.

\item If $d\nmid m$ (necessarily $d$ even) then $a^{m} \notin H$ and $|H\cap K|=1$. Thus $|HK|=4(2m/d)=8m/d$. Setting this equal to $4m$ forces $d=2$, that is $H=\langle a^{2}\rangle$.
\end{itemize}

Reciprocally, if $H$ is either $\langle a\rangle$ or $\langle a^{2}\rangle$ and $K= \langle a^{k}b \rangle$, then $|HK|=4m$; hence $HK=Q_{4m}$.
\end{proof}

\smallskip

Let us again denote with $\Gamma'(Q_{4m})$ the subgraph obtained from $\Gamma(Q_{4m})$ by removing the isolated vertices.

\begin{corollary}\label{Dicyclic}
Let $m \geq 3$ be an odd integer. Then the reduced factorization graph $\Gamma'(Q_{4m})$ satisfies the following:
\begin{enumerate}
    \item The graph $\Gamma'(Q_{4m})$ is isomorphic to the complete 
          bipartite graph $K_{2,m}$.
    \item $\Gamma'(Q_{4m})$ is planar.
    \item The graph $\Gamma'(Q_{4m})$ has a maximal induced star 
          with central vertex $\langle a \rangle$ or $\langle a^2 \rangle$ 
          and $m$ leaves.
    \item The total domination number of $\Gamma'(Q_{4m})$ is 
      $\gamma_t(\Gamma'(Q_{4m})) = 2$. Moreover, any set 
      consisting of one vertex from $X$ and one vertex from 
      $Y$ is a minimum total dominating set.
    \item $\mathrm{girth}(\Gamma'(Q_{4m})) = 4$.
    \item The order of $\Gamma'(Q_{4m})$ is $m+2$ and its size is $2m$.
    \item $\alpha(\Gamma'(Q_{4m})) = m$.
    \item $\Gamma'(Q_{4m})$ is non-Eulerian but semi-Eulerian.
    \item $\Gamma'(Q_{4m})$ is non-Hamiltonian.
    \item $\Gamma'(Q_{4m})$ is a perfect graph.
    \item The center of $\Gamma'(Q_{4m})$ is the complete vertex set, 
          and $\mathrm{diam}(\Gamma'(Q_{4m})) = 2$.
\end{enumerate}
\end{corollary}

\begin{proof}
Let $X = \{\langle a^k b \rangle \mid 0 \leq k < m\}$ and 
$Y = \{\langle a \rangle, \langle a^2 \rangle\}$.

\medskip
\noindent (1) By Theorem \ref{subq4}, two distinct proper 
nontrivial subgroups $H, K$ of $Q_{4m}$ satisfy $HK = Q_{4m}$ if and 
only if one belongs to $Y$ and the other to $X$. Since 
$X \cap Y = \emptyset$, $|X| = m$, and $|Y| = 2$, the graph 
$\Gamma'(Q_{4m})$ is the complete bipartite graph $K_{2,m}$ 
with parts $X$ and $Y$.

\medskip
\noindent Items (2) - (11) are direct consequences of (1) and standard 
properties of $K_{2,m}$.
Specifically: $K_{2,m}$ is outerplanar (2); each part induces a 
maximal star (3); the set $\{\langle a\rangle, \langle b\rangle\}$, 
with one vertex from each part, is a minimum total dominating 
set (4); the girth is $4$ since $K_{2,m}$ is bipartite and 
contains $4$-cycles (5); the order is $m+2$ and the size 
is $2m$ (6); the independence number is 
$\alpha = m$ (7); exactly two vertices have odd degree $m$, 
so the graph is semi-Eulerian but not Eulerian (8); 
$K_{2,m}$ is non-Hamiltonian since $|X| \neq |Y|$ (9); 
every bipartite graph is perfect (10); and a direct case 
analysis on pairs from $X$, $Y$, or mixed gives 
$\mathrm{diam} = 2$ with every vertex achieving 
eccentricity $2$ (11).
\end{proof}

\begin{theorem}\label{puntosaisladosII}
Let $m\geq 3$. Then the unique isolated point of $\Gamma(Q_{4m})$  is the subgroup $\langle a^{m}\rangle$.
\end{theorem}

\begin{proof}
Let $A=\langle a^m\rangle$.
\smallskip
\noindent\emph{Claim 1: $A$ is isolated.}
Let $H$ be any proper nontrivial subgroup of $G$. We distinguish two cases.

\smallskip
\noindent (i) $H\leq \langle a\rangle$.
Then $AH\leq \langle a\rangle\neq G$, hence $AH\neq G$.

\smallskip
\noindent (ii) $H\not\leq \langle a\rangle$.
Then $H$ contains an element of the form $a^k b$. As in the standard subgroup description for $Q_{4m}$,
this forces $H=\langle a^k b\rangle$, and in particular $(a^k b)^2=a^m\in H$, so $A\leq H$.
Thus $AH=H\neq G$. In both cases, $A$ is adjacent to no vertex, so it is isolated.

\smallskip
\noindent\emph{Claim 2: no other vertex is isolated.}
Let $H\neq \langle a^m\rangle$ be a proper nontrivial subgroup. If $H\leq \langle a\rangle$, then
$H$ is either $\langle a\rangle$ or $\langle a^2\rangle$, and we have $G=\langle a\rangle\langle b\rangle
=\langle a^2\rangle\langle b\rangle$, so $H\sim \langle b\rangle$.
If $H=\langle a^k b\rangle$, then $G=\langle a\rangle \langle a^k b\rangle$, so $H\sim \langle a\rangle$.
Hence, every vertex different from $\langle a^m\rangle$ has a neighbor, and therefore it is not isolated.
Consequently, the unique isolated vertex of $\Gamma(G)$ is $\langle a^m\rangle$.
\end{proof}

\section{Conflict of interest}
On behalf of all authors, the corresponding author states that there is no conflict of interest.

\appendix

\section{An algorithm}
The following algorithm generates the factorization graph of an arbitrary finite group $G$ by computing its subgroups and adding an edge between each pair of co-maximal subgroups.

\smallskip

\begin{algorithm}[H]\label{algo1}
\KwData{A finite group $G$}
\KwResult{$\Gamma(G)$}
\begin{verbatim}
def FACT_GRAPH(G):
    U = Graph()
    VU = [H for H in G.subgroups() if 1 < H.order() < G.order()]

    for i, H in enumerate(VU):
        h_label = str(H.gens()).replace('[', '<').replace(']', '>')
        U.add_vertex(h_label)

        for j, K in enumerate(VU):
            if H != K:
                HK = {h * k for h in H for k in K}
                if HK == set(G):  
                    k_label = str(K.gens()).replace('[', '<').replace(']', '>')
                    U.add_edge(h_label, k_label)
    U.show()
\end{verbatim}
\caption{A \textit{Python} function that computes the factorization graph of a finite group $G$ in \textit{SageMath}.}
\end{algorithm}

\end{document}